\documentclass[12pt]{article}
\usepackage{amsmath,amsthm,amsfonts,amssymb,verbatim}
\usepackage[all]{xy}
\usepackage[usenames]{color} 
\newtheorem{thm}{Theorem}[section]
\newtheorem{lemma}[thm]{Lemma}

\newtheorem{prop}[thm]{Proposition}

\theoremstyle{definition}

\newtheorem{remark}[thm]{Remark}

\theoremstyle{remark}
\DeclareMathOperator{\Fix}{Fix}

\DeclareMathOperator{\Homeo}{Homeo}

\DeclareMathOperator{\dist}{dist}

\newcommand{\R}{\mathbb R}

\newcommand{\F}{{\cal F}}

\newcommand{\Z}{\mathbb Z}

\def\ti{\tilde}

\def\sl3z{SL(3, \mathbb Z)}

\def\L{{\cal L}}
\def\M{{\cal M}}

\title{Complete semi-conjugacies for psuedo-Anosov homeomorphisms}
\author{John Franks\thanks{Supported in part by NSF grant DMS-055463.} ,\ \ 
Michael Handel\thanks{Supported in part by NSF grant DMS0103435.}}

\begin{document}
\maketitle
\begin{abstract}

Suppose $S$ is a surface of genus $\ge 2 $, \ $f: S \to S$ is a surface
homeomorphism isotopic to a pseudo-Anosov map $\alpha$ and suppose
$\ti S$ is the universal cover of $S$ and $F$ and $A$ are lifts of $f$
and $\alpha$ respectively.  A result of A. Fathi shows
there is a semiconjugacy $\Theta :
\ti S \to \bar \L^s \times \bar \L^u$ from $F$ to $\bar A$, where
$\bar \L^s$ ($\bar \L^u$) is the completion of the $R$-tree of leaves
of the stable (resp. unstable) foliation for $A$ and $\bar A$ is the
map induced by $A$.

We generalize a result of Markovich and show that   for any 
$g \in \Homeo(S)$  that commutes with $f$ and is isotopic to the identity with identity lift $G$ and for any $(c,w)$ in the image of $\Theta$
each component of $\Theta^{-1}(c,w)$ is $G$-invariant.
\end{abstract}

\section{Introduction}

Suppose that $S$ is a closed surface and that $\alpha:S \to S$ is
either an orientation preserving linear Anosov map of $T^2$ or an
orientation preserving pseudo-Anosov homeomorphism of a higher genus
surface.  In the former case note that $\alpha$ fixes the point $e$
that is the image of $(0,0)$ in the usual projection of $\R^2$ to
$T^2$.  The first author \cite{franks:semiconjugacy} proved that if $f
\in \Homeo(T^2)$ is isotopic to $\alpha$ and fixes $e$ then there is a
unique map $p:T^2 \to T^2$ that fixes $e$, is isotopic to the identity
and that semi-conjugates $f$ to $\alpha$; i.e $pf = \alpha p$.
   
To describe this case further we work in the universal cover $\R^2$ of
$T^2$.  Let $A,F$ and $P$ be the lifts of $ \alpha,f$ and $p$
respectively that fix $(0,0)$ and note that $PF^k = A^kP$ for all $k
\in \Z$ (because $PF^k((0,0)) = A^kF((0,0)$). We say that the $F$-orbit
of $\ti y \in \R^2$ {\em shadows} the $A$-orbit of $\ti x \in \R^2$ if
$\ti d(F^k(\ti y), A^k(\ti x)) \le C$ for some $C$ and all $k \in \Z$;
we also say that the $f$-orbit of $y \in T^2$ {\em globally shadows}
the $\alpha$-orbit of $x \in T^2$.  Since $p$ is homotopic to the
identity, there exists $C > 0$ so that $\dist(P(\ti y), \ti y) < C$
for all $\ti y \in \R^2$.  In particular, $\dist(F^k(\ti y), A^k(P(\ti y)) =  \dist(F^k(\ti y), PF^k(\ti y)) < C$ for all $k$ so   the $F$-orbit of $\ti y$
shadows the $A$-orbit of $P(\ti y)$ for all $\ti y \in \R^2$.  It is
well known that no two $A$-orbits shadow each other so $P$ is
completely determined by this shadowing property.  The surjectivity of
$P$ reflects the fact that every $A$-orbit is shadowed by some
$F$-orbit.  The fact that $P$ is defined on all of $\R^2$ reflects the
fact that every $F$-orbit shadows some $A$-orbit.
   
Suppose now that $\alpha$ is pseudo-Anosov, that $f$ is isotopic to
$\alpha$ and that $A: \ti S \to \ti S$ is a lift of $\alpha$.  The
isotopy between $\alpha$ and $f$ lifts to an isotopy between $A$ and a
lift $F : \ti S\to \ti S$ of $f$.  Equivalently, $F$ is the unique
lift of $f$ that induces the same action on covering translations as
$A$.  Let $d(x,y)$ be any path metric on $S$ and let $\ti d(\ti x, \ti
y)$ be its lift to $\ti S$.  Shadowing in $\ti S$ and global shadowing
in $S$ are defined as above using $\ti d$ in place of the Euclidean
metric on $\R^2$.  It is not hard to construct examples
(c.f. Proposition~ 2.1 of \cite{han:entropy}) for which there are
$F$-orbits that are not shadowed by any $A$-orbit.  On the other hand, the second author proved  \cite{han:semiconjugacy} that
every $A$-orbit is shadowed by some $F$-orbit.  More precisely, there exists a closed $f$-invariant set $Y \subset S$  with full pre-image denoted $\ti Y \subset \ti S$  
and a continuous equivariant  surjection $P: \ti Y \to \ti S$ such that $PF = AP$.  

The $F$-orbits of $\ti S \setminus \ti Y$ do not shadow $A$-orbits.
It is natural to ask if there is some larger context in which one can
understand these orbits.  Fathi \cite{Fa} answered this  by considering leaf
spaces of the stable and unstable foliations.

In the Anosov case, the hyperbolic linear map $A$ has stable and
unstable invariant foliations $\F^s$ and $\F^u$ consisting of straight
lines parallel to the eigenvectors of $A$.  Let $\L^s$ and $\L^u$ be
the leaf spaces of $\F^s$ and $\F^u$ respectively.  Then $\L^s$ and
$\L^u$ can be identified with $\R$ in such a way that $A$ induces
homotheties $A^s :\L^s \to \L^s$ and $A^u :\L^u \to \L^u$ defined
respectively by $x \to \lambda x$ and by $x \to x/\lambda$ where
$\lambda>1$ is an eigenvalue of $A$.  Thus $(A^s,A^u):\L^s \times \L^u
\to \L^s \times \L^u$ is naturally identified with $A : \R^2 \to \R^2$
and the shadowing argument that defines $P$ can be done in this
product of leaf spaces.

In the pseudo-Anosov case the stable and unstable foliations $\F^s$
and $\F^u$ have singularities and their leaf spaces $\L^s$ and $\L^u$
are more complicated. Namely, $\L^s$ and $\L^u$,
and their metric completions $\bar \L^s$ and $\bar \L^u$ have the
structure of $\R$-trees \cite{morganShalen}.  As in the previous case $A$ induces a
homothety $A^s :\bar \L^s \to \bar \L^s$ that uniformly expands
distance by a factor $\lambda>1$ and a homothety $A^u :\bar \L^u \to
\bar \L^u$ that uniformly contracts distance by the factor
$1/\lambda$.  Let $\bar A = (A^s,A^u ) :\bar \L^s \times \bar \L^u \to \bar \L^s \times \bar \L^u$.
 
Let $Q _s: \ti S \to \L^s$ be the natural map that sends a point in
$\ti S$ to the leaf of $\F^s$ that contains it. Define $Q _u$
similarly and let $Q=Q_s \times Q_u : \ti S \to \L^s \times \L^u$.
Denote the subset of $\L^s \times \L^u$ consisting of pairs of leaves
of $\F^s$ and $\F^u$ that have a point in common by $\Delta$.    Then $Q(\ti S)=\Delta$, and $Q : \ti S \to
\Delta$ is a homeomorphism.  Moreover, the following diagram
 commutes. $$
 \xymatrix{
\ti Y   \ar[rr]^F \ar[dd]_P && \ti Y \ar[dd]^P\\&&\\
\ti S   \ar[rr]^A \ar[dd]_Q&& \ti S \ar[dd]^Q\\&&\\
 \Delta \ar[rr]_{\bar A|_\Delta} &&  \Delta \
}
$$

 Fathi  \cite{Fa} extended  $QP$ to a non-surjective
map $\Theta: \ti S \to \bar \L^s \times \bar \L^u$ that makes the following diagram commutes.
$$
 \xymatrix{
\ti S   \ar[rr]^F \ar[dd]_\Theta && \ti S \ar[dd]^\Theta\\&&\\
\bar \L^s \times \bar \L^u \ar[rr]_{\bar A} && \bar \L^s \times \bar \L^u \
}
$$
 See also \cite{RHU} for more details on the map $\Theta$.

 The maps $p :T^2
\to T^2$ and $\Theta:S \to \bar \L^s \times \bar \L^u$ depend
canonically on $f$ and so determine canonical decompositions
$\{p^{-1}(x)\}$ of $T^2$ and $\{\Theta^{-1}(c,w)\} $ of $S$.  If $g$
commutes with $f$ then one expects $g$ to preserve this decomposition.
If $g$ is isotopic to the identity then one might even expect $g$ to
setwise preserve each element of the decomposition. 

Implicit in \cite{mark} is an even stronger and more
surprising fact for    the Anosov
case.  Namely that if $g$ commutes with $f$ and is isotopic
to the identity then $g$ setwise preserves each {\em component} of the
decomposition.  It is not hard to see that this decomposition is upper
semi-continuous and that each element is cellular. 
       
  \begin{thm} {\bf (Markovich)} \label{Markovich}  
Suppose that $\alpha:T^2 \to T^2$ is a linear Anosov map, that $f\in
\Homeo(T^2)$ is isotopic to $\alpha$ and fixes $e$ and that $p:T^2 \to
T^2$ is the unique map that fixes $e$ and satisfies $pf = \alpha p$.
If $g \in \Homeo(T^2)$ commutes with $f$, is isotopic to the identity
and setwise preserves $p^{-1}(e)$ then each component of $p^{-1}(x)$ is $g$-invariant for all $x \in T^2$.
\end{thm}

The main result of this paper is the following extension of
Markovich's result to the pseudo-Anosov case.  If $g$ is isotopic to
the identity then the {\em identity lift} $G$ of $g$ is the unique
lift that commutes with all covering translations.

\begin{thm} \label{pA Markovich}  
Suppose that $\alpha:S\to S$ is pseudo-Anosov, that $f\in \Homeo(S)$
is isotopic to $\alpha$ and that $A,F : \ti S \to \ti S$ and   $\Theta:\ti S \to \bar \L^s \times \bar \L^u$ are as above.   If $g \in \Homeo(S)$ commutes
with $f$ and is isotopic to the identity then the identity lift $G : \ti S \to \ti S$ of $g$ commutes with $F$ and preserves each component of
$\Theta^{-1}(c,w)$  for all $(c,w) \in \Theta(\ti S)$.  \end{thm} 
 
     Our
proof of Theorem~\ref{pA Markovich}  makes use of arguments from
\cite{mark}.  We give complete details for the reader's convenience and   because the arguments from
\cite{mark} are not easily referenced.     It is straightforward to modify our proof of  Theorem~\ref{pA Markovich} to obtain a proof of   Theorem~\ref{Markovich}.  We leave that to the interested reader.

\section{Proof of Theorem~\ref{pA Markovich}}

We assume throughout this section  that  $\alpha:S \to S$ is   pseudo-Anosov with expansion factor $\lambda >1$,    that $f \in \Homeo(S)$ is isotopic to $\alpha$, and that $A: \ti S \to \ti S$ and $F: \ti S \to \ti S$ are   lifts of $\alpha$ and $f$ that induce the same action on covering translations.  
 
The transverse measures on the lifts $\ti \F^s$ and $\ti \F^u$  of  the stable and  unstable measured foliations  $\F^s$ and $\F^u$ for $\alpha$   determine  pseudo-metrics $d^u$  and $d^s$ on $\ti S$ such that 
\begin{equation}\label{eqn1}
d^u(A(x),A(y))  =  \lambda d^u(x,y) \text{ and }
d^s(A(x),A(y)) = \lambda^{-1}d^s(x,y).
\end{equation}
for  all  $x,y  \in \ti  S$.  Moreover,  $d^u(x,y) = 0$ [resp. $d^s(x,y) = 0$]  if and only if $x$ and $y$ belong to the same leaf of $\ti \F^s$ [resp. $\ti \F^u$].   There is also \cite{bers} a singular Euclidean metric $d$ on $\ti S$ 
such that
\[
d(x,y) = \sqrt{d^s(x,y)^2 + d^u(x,y)^2}
\]
 on each standard Euclidean
chart (i.e. one without singularities).   In particular,
$d(x,y)\le 2 \max\{d^s(x,y),d^u(x,y)\}$.
 
For any $x,y \in \ti S$ there is a unique (up to parametrization)
path $\rho$ with endpoints $x$ and $y$ and with length equal to
$d(x,y)$.  Subdividing $\rho$ at the singularities that it intersects
decomposes $\rho$ into a concatenation of linear subpaths.  We will
refer to $\rho$ as the {\em geodesic joining $x$ to $y$}.  If the
intersection of two geodesics $\gamma_1$ and $\gamma_2$ is more than a single point, then it is a path whose endpoints are either singularities or endpoints of  $\gamma_1$ or $\gamma_2$.    Any leaf without singularities of
either foliation is a geodesic, as is any embedded copy of $\R$ in a
leaf with singularities.

\begin{remark}  \label{du smaller than d}Since the geodesic  
$\gamma$ joining $x$ to $y$ (minimizing length as measured
by $d$) consists of a finite collection of segments whose
interiors lie in Euclidean charts we can conclude
$\max\{d^s(x,y),d^u(x,y)\} \le d(x,y) \le 2 \max\{d^s(x,y),d^u(x,y)\}$.  
\end{remark}

We denote the leaf of $\ti \F^s$ that contains $x$ by $W^s(x)$, the leaf space of $\ti \F^s$ by $\L^s$  and the image of $W^s(x)$ in $\L^s$ by $Q_s(x))$.  Thus $Q_s^{-1}(Q_s(x)) = W^s(x)$.       Moreover,
$d^u$ induces a metric on $\L^s$ (which we also denote $d^u$)
by setting $d^u(Q_s(x_1), Q_s(x_2)) = d^u(x_1, x_2)$.  This
is easily seen to be independent of the choice of
$x_1$ and $x_2.$  
This metric gives $\L^s$ 
the structure of an $\R$-tree \cite{morganShalen}.  As illustrated in Remark~\ref{incomplete}, $\L^s$ is not complete.  We denote the metric completion by $(\bar \L^s, d^u)$

\begin{remark}  \label{incomplete} One can construct a non-converging Cauchy sequence $L_i =Q_s(x_i)$  in $\L^s$ as follows.  Choose a singularity  $\ti x_0 \in \ti S$,  a stable ray  $R_0$ initiating    at $x_0$ and a sequence $\epsilon_i >0$ whose sum is finite.   Assuming inductively that $x_i$ and $R_i$ have been defined, choose a singularity  $x_{i+i} \in \ti S$  such that $W^u(x_{i+1}) \cap R_i \ne \emptyset $ and such that $d^u(x_{i+1},x_i) < \epsilon_{i+1}$.  Let $R_{i+1}$ be a  stable ray  in $W^s(x_{1+1})$ that  initiates at $x_{i+1}$  and whose interior is     contained in a  component of the complement of $W^u(x_{i+1})$ that is disjoint from $x_i$.    By construction, $d^u(L_i, L_j)  = \epsilon_{i+1} + \dots + \epsilon_j$  for all $i \le j$  and there are no accumulation points of the $W^s(x_i)$'s.   The former shows that  the $L_i$'s are a Cauchy sequence and the latter implies that this sequence is non-convergent.
\end{remark}

 Similarly $\L^u$, the space
of unstable leaves in $\ti S$ is an $\R$-tree with metric
$d^s$   and $Q_u : \ti S \to \L^u$ satisfies $Q_u^{-1}(Q_u(x)) = W^u(x)$.   The metric completion of $\L^u$  is  $(\bar \L^u, d^s)$ and we use the metric  $\bar d = \max\{d^u, d^s\}$ on $ \bar \L^u \times \bar \L^s$.
Define $Q : \ti S \to \L^u \times \L^s$   by $Q(x) =
(Q^s(x), Q^u(x))$  and let $\Delta = Q (\ti S)$ equipped with the subset topology.   Note that $\Delta$ can be characterized as the subset  of $\L^s \times \L^u$ consisting
of pairs of leaves of $\F^s$ and $\F^u$ that have a point in common.

The pseudo-Anosov map $A$ on $\ti S$ induces a homeomorphism
$A_s : \L^s \to \L^s.$  From the properties of $d^s$ (equation 
(\ref{eqn1}) above) it is
clear that $A_s$ is an expanding homothety, i.e. for any
$L_1, L_2 \in \L^s$
\begin{equation}\label{eqn7}
d^u(A_s(L_1), A_s(L_2)) = \lambda d^u(L_1, L_2).
\end{equation}
The map $A_u : \L^u \to \L^u$ is defined analogously and has similar
properties.  Likewise if $T$ is a covering translation on $\ti S$ it
induces an isometry $T_s : \L^s \to \L^s.$ It is clear that $A_s$ and
$A_u$ extend to expanding and contracting homotheties respectively of
$\bar \L^s$ and $\bar \L^u$ which we also denote by $A_s$ and $A_u$.
Also $T_s$ extends to an isometry of $\bar \L^s.$ The map $T_u : \L^u
\to \L^u$ is defined analogously and has similar properties.  Let
$\bar A : \bar \L^s \times \bar \L^u \to \bar \L^s \times \bar \L^u$
be the map $A_s \times A_u$.

As a consequence of Fathi's theorem we have the following.

\begin{prop} \label{leaf_semiconj} 
There exists a unique continuous map
$\theta_s: \ti S \to \bar \L^s$ satisfying
 \begin{enumerate}
\item [(i)] There is a commutative diagram
\[
 \xymatrix{
\ti S   \ar[rr]^F \ar[dd]_{\theta_s} && \ti S \ar[dd]^{\theta_s}\\&&\\
\bar \L^s  \ar[rr]_{A_s} && \bar \L^s \
}
\]
\item [(ii)]  There exists $C_1>0$ such that
$d^u(Q_s(x),\theta_s(x)) <C_1$ for all $x \in \ti S.$
 \item [(iii)]  For all $x \in \ti S,$
\ $\theta_s(x)$ is the unique point 
in $\bar \L^s$ with
the property that there exists $C_2>0$ such that
$d^u(Q_s(F^k(x)),A_s^k \theta_s(x)) <C_2$ for all $k \ge 0.$
\item [(iv)]   $\theta_s(T(x)) = T_s(\theta_s(x))$ for all $x
\in \ti S$ and each covering transformation $T$.
\item [(v)] The image of $\theta_s$ contains $\L^s.$
 \end{enumerate}
\end{prop}

\begin{proof}  
Let $\theta_s : \ti S \to \bar \L^s$ be the composition of
$\Theta$ with the projection of  $\bar \L^s \times \bar \L^u$
onto its first factor.  Then properties (i)-(v) follow from 
the corresponding properties of $\Theta$ as proved in \cite{Fa}.
\end{proof}

Let $\theta_u : \ti S \to \bar \L^u$ be the analogous semiconjugacy
from $F$ to $A_u : \bar \L^u \to \bar \L^u.$  Note then that
\[
\Theta(x) = (\theta_u(x), \theta_s(x)).
\]

\begin{remark} Proposition~\ref{leaf_semiconj} can be proved directly as follows.  Given $x \in \ti S$ let $z_k = A^{-k}F^k(x)$
and let $L_k = Q_s(z_k)$.  Then it is straight forward to show that
$\{L_k\}$ is a Cauchy sequence in the $d^u$ metric on $\bar \L^s$.
It's limit can be taken as the definition of $\theta_s(x)$.  Property
(i) is then immediate and it is not difficult to show $\theta_s$
is continuous and satisfies the other properties.
\end{remark}

We denote  open $\epsilon$-neighborhoods by  $N_{\epsilon}(\cdot)$. 
   By item (iii) of Proposition~\ref{leaf_semiconj} applied with $k =0$ there is a constant $C_2$ such that 
 \begin{equation}\label{secondExtended}
 d^u(Q_s(x), \theta_s(x)) <C_2
\end{equation}
 for all $x \in \ti S$.      Since $d^u(A(x), F(x))$ is bounded independently of $x$ and since $Q_s$ preserves $d^u$ we may also assume that
  \begin{equation}\label{fifthExtended}
d^u(Q_sF(x), A_sQ_s(x)) = d^u(Q_sF(x),Q_sA(x)) < C_2.
\end{equation}  
    Choose  $C > \max(C_2, C_2/(\lambda-1))$ and note that 
\begin{equation}\label{thirdExtended}
\lambda C - C_2 > C
\end{equation}
     
\begin{lemma}\label{bounded}  
$\Theta^{-1} (N_{\epsilon}(c) \times N_{\epsilon}(w))$ is a
bounded subset of $\ti S$ for all $\epsilon > 0$ and all $(c,w) \in
\L^s \times \L^u$.
\end{lemma}

\proof For $x \in Q_s^{-1}(c)$ and $z \in \Theta^{-1} (N_{\epsilon}(c)
 \times N_{\epsilon}(w))$ we have
\begin{align*}
d^u(x,z) 
&= d^u(Q_s(x),Q_s(z))\\
& \le
d^u(Q_s(x),\theta_s(z)) 
  + d^u(\theta_s(z),Q_s(z))\\
&\le  \epsilon +C_2.
\end{align*} 
   Symmetrically, if $\theta_u(y) = w$ then 
$$
d^s(y,z) \le  \epsilon +C_2.
$$ 
It follows that $d(z,z')\le 2 \max\{d^u(z,z'), d^s(z,z')\} 
\le 2(\epsilon + C_2)$ for all 
$z,z' \in \Theta^{-1} (N_{\epsilon}(c)
\times N_{\epsilon}(w))$.  \endproof
  
\begin{lemma} \label{epsilon neighborhood is nice}   
If $V(c,\epsilon)$ is defined to be $\theta_s^{-1}(N_\epsilon(c))$, then
\begin{enumerate}
\item [(i)] The set $V(c,\epsilon)$ is an
open, connected, simply connected, unbounded set for all $c\in \L^s$
and all $\epsilon > 0$.  
\item [(ii)] If $G: \ti S \to \ti S$ commutes with $F$ and
if there is a constant $C_1$ such that $d(x, G(x)) < C_1$ for all $x
\in \ti S$ then there is a ray $R$ that is properly embedded  in $\ti
S$ such that $R, G(R) \subset V(c,\epsilon)$ and such that $R$ is
properly homotopic to $G(R)$ in $V(c,\epsilon)$; i.e. there is a one
parameter family $R_t$ of rays in $V(c,\epsilon)$ such that $R_0 = R$
and $R_1 = G(R)$ and such that each $R_t$ is properly embedded in $\ti
S$.
\end{enumerate}
\end{lemma}

\proof    Define
$$
Y_k = \{x \in \ti S:  d^u( x, W^s(A^k(c)) < \lambda^k \epsilon - C\}
$$ 
or equivalently
$$
Y_k = Q_s^{-1}(N_{\lambda^k \epsilon - C}( A_s^k(c))).
$$ 
Then $Y_k$ is an open convex subset of $\ti S$ that is a union of
leaves of $\F^s$.  In particular, $Y_k$ is an open, connected, simply
connected, unbounded set. If $R_1(t)$ and
$R_2(t)$ are any two rays in $Y_k$ that are properly
embedded in $\ti S$ and if there is a constant $C_0$ such
that $d(R_1(t),R_2(t)) \le C_0$ for all $t$, then these rays
are properly homotopic in $Y_k$ (by a homotopy along geodesics).
 
Define
$$
X_k = F^{-k}(Y_k).
$$  
Thus
 \begin{align*}
X_k &= F^{-k}Q_s^{-1}(N_{\lambda^k \epsilon - C}( A_s^k(c)))\\
& =  \{x \in \ti S:  d^u(Q_sF^k(x), A_s^k(c)) < \lambda^k \epsilon - C\}.
\end{align*}
Since $F$ is a homeomorphism, each $X_k$ is an open, connected, simply
connected, unbounded set. Also if $R_1(t)$ and $R_2(t)$ are
properly embedded rays in $\ti S$, contained in $X_k$, 
for which there is a constant $C_0$ such that
$d(R_1(t),R_2(t)) \le C_0$, then these rays are properly homotopic in
$X_k$.

 If $x \in X_k$ then Equation~(\ref{secondExtended})  implies that $A_s^k\theta_s(x) = \theta_sF^k(x) \in N_{\lambda^k\epsilon}(A_s^k(c))$   and  so $\theta_s(x) \in N_\epsilon(c)$ by  Equation~(\ref{eqn7}).   This proves that $X_k \subset V(c,\epsilon)$.    Moreover, by the triangle inequality and Equations~(\ref{fifthExtended}), (\ref{eqn7}) and (\ref{thirdExtended}) we have
\begin{align*}
d^u(Q_sF^{k+1}(x), A_s^{k+1}(c)) &= d^u(Q_sF F^{k}(x), A_sA_s^{k}(c))\\
&\le d^u(A_sQ_s F^{k}(x), A_sA_s^{k}(c)) + C_2\\
&= \lambda d^u(Q_s F^{k}(x),A_s^{k}(c)) + C_2\\
&\le \lambda(\lambda^k\epsilon - C) + C_2\\
&= \lambda^{k+1} \epsilon -(\lambda C - C_2)\\
&\le  \lambda^{k+1} \epsilon - C
\end{align*}
which proves that $X_k \subset X_{k+1}$.

If  $\theta_s(w)\in N_\epsilon(c)$ then we may choose  $\delta < \epsilon$ and $k \ge 1$ so that  $\theta_s(w)\in N_\delta(c)$ and  so that     $\lambda^k\delta < \lambda^k \epsilon -C_2 -C$. 
 Then 
 $$
 \theta_sF^k(w)
= A_s^k\theta_s(w) \in N_{\lambda^k \epsilon -C-C_2}(A_s^k(c))
$$
 and  Equation~(\ref{secondExtended}) implies that     
 $$
 Q_sF^k(w)   \in N_{\lambda^k \epsilon-C}(A_s^k(c))
 $$
 and hence that $w \in X_k$.     We have now shown that
$V(c,\epsilon)$ is the increasing union  of open, connected, simply connected and unbounded sets $X_k$ thereby  completing
the proof of (i).

Choose $k \ge 1$ so that $C_1 < \lambda^k \epsilon -C$ and let $R'$ be
a ray in $W^s(A_s^k(c))$.  Then $R'$ and $ G(R') )$ are properly
embedded in $\ti S$ and are contained in $Y_k$ and $R := F^{-k}(R')$
and $G(R) = F^{-k}G(R')$ are properly embedded in $\ti S$ and are
contained in $X_k$.  Choosing $C_0$ so that  $d(G(x),x) \le C_0$ for all 
$x$, we conclude $d(R(t), G(R(t))) \le C_0$ and
as noted above, $R$ and $G(R)$ are properly
homotopic in $X_k \subset V(c,\epsilon)$.  This proves (ii).
\qed
 
\vspace{.2in}

  The following proof is an adaptation  of one that appears in \cite{mark}.

\vspace{.2in}

\noindent{\bf Proof of Theorem~\ref{pA Markovich} }  Let $G : \ti S \to \ti S$ be the unique lift of $g$ that   is equivariantly isotopic to the
identity; equivalently $G$ is the lift that  commutes with all covering translations $T$ of $\ti S$.
The commutator $[F,G]$, which must be a covering translation
since it is a lift of the identity on $S$, commutes with all covering translations and hence is the identity.  It follows that  $F$ and $G$
commute.

 Since $S$ is compact and  $\bar d(Q(G(Ty)),Q(Ty))  = \bar d(Q(G(y)),Q(y))$ for all $y \in \ti S$ and    all covering translations $T$,   there is a
constant $C'$ such that $\bar d(Q(G(y)),Q(y)) < C'$ for all $y \in \ti
S$.   It follows that
    
\begin{align*}
\bar d(QF^k(x), \bar A^k(\Theta(G(x)))
&\le \bar d(QF^k(x), QF^k(G(x))  +  \bar d(QF^k(G(x)),  \bar A^k\Theta( G(x)))\\
&= \bar d(QF^k(x), QG(F^k(x))  +  \bar d(QF^k(G(x)),  \bar A^k\Theta( G(x)))\\
& \le C' +    \bar d(QF^k(G(x)),  \bar A^k\Theta( G(x)))\\
&   \le C' + \overline{C_2}
\end{align*}
where $\overline{C_2}$ is the maximum of the constants produced by item (iii)  of Proposition~\ref{leaf_semiconj} applied to $\theta_s$ and to $\theta_u$.
The uniqueness part of Proposition~\ref{leaf_semiconj}-(iii) therefore
implies that $\Theta G = \Theta$.  In particular, $\Theta^{-1}(c,w)$
is $G$-invariant.  It suffices to show that each component of
$\Theta^{-1}(c,w)$ is $G$-invariant.

Choose $\epsilon_n \to 0$. Let $V_n = V(c,\epsilon_n)$ be as in
Lemma~\ref{epsilon neighborhood is nice} and let $H_n = H(w,\epsilon_n)$ be the
open set obtained by applying Lemma~\ref{epsilon neighborhood is nice} with
$\L^s$ replaced by $\L^u$ and $\theta_s$ replaced by $\theta_u$.  Then  $V_n \cap  H_n = \Theta^{-1}(N_{\epsilon}(c) \times N_{\epsilon}(w))$  is bounded by Lemma~\ref{bounded} and
\[
\Theta^{-1}(c,w)  = \cap_{n=1}^{\infty}(V_n \cap  H_n).
\]
Moreover if $\Lambda$ is a component of $\Theta^{-1}(c,w)$ and we let
$K_n$ be the component of $V_n \cap H_n$ containing $\Lambda$ then
$\overline{K_{n+1}} \subset K_n$ for all $n$ where
$\overline{K_{n+1}}$ denotes the closure of $K_{n+1}$.  In particular,
$$
\cap_{n=1}^{\infty} K_n = \cap_{n=1}^{\infty} \overline{K_{n+1}}.
$$ 
Since each  $\overline{K_{n+1}}$ is compact,  $\cap_{n=1}^{\infty} K_n$ is  non-empty and connected and hence equal to $\Lambda$.
 It therefore suffices to
show that each  $K_n$   is $G$-invariant.    In fact it is enough to prove that  $G(\bar K_n) \cap \bar K_n \ne
\emptyset$, because then
$G(\bar K_n)$ and $\bar K_n$ are both subsets of $K_{n-1}$ and
hence $G(K_{n-1}) = K_{n-1}$.

If there were infinitely many components of   $V_n \cap H_n$ that contain an element of $\Theta^{-1}(c,w)$ then there
would be a sequence $\{x_i\} \subset    \Theta^{-1}(c,w)$ converging to some $x \in
\Theta^{-1}(c,w)$ and with each $x_i$ in a
different component of $V_n \cap H_n$.  Since these components are
disjoint, $x$ is also an accumulation point of the frontiers of these
components.  Hence, since the frontier of a component of $V_n \cap
H_n$ is contained in the union of the frontier of $V_n$ and the
frontier of $H_n$ there is a sequence $\{y_i\}$ converging to $x$ with
each $y_i$ in the frontier of either $V_n$ or $H_n$.  This contradicts the  continuity of $\Theta$ and the fact that $\bar d(\Theta(y_i), \Theta(x_i)) = \epsilon_n > 0$ for all $i$.  
We conclude that there are only finitely many components of   $V_n \cap H_n$ that contains an element of $\Theta^{-1}(c,w)$.  Since $G$ commutes with $\Theta$, $G^i(K_n)$ is a such a component for all $i \ge 0$.  Thus  $G^m(K_n) = K_n$
  for some smallest $m\ge 1$.

Let $S^2$ denote the one point compactification of $\ti S$ obtained
by adding a point $\infty$.  The set $V_n \subset \ti S$ can be thought  of as a subset of $S^2$ and when we do so we refer to it simply as $V$.      By   Lemma~\ref{epsilon neighborhood is nice}, $V$ is    open, connected and simply
connected and so has a     prime 
end compactification $\widehat V$.   
 For our purposes the key properties 
are:

\begin{itemize}
\item [(i)]  $\widehat V$ is topologically a disk $D$ whose interior is identified with
$V$.
\item [(ii)] The function $G|_V$ extends continuously to a homeomorphism
$\widehat G : D \to D.$
\item[(iii)] For each continuous arc $\gamma: [0,1] \to  S^2$ with
$\gamma([0,1)) \subset V$ and $\gamma(1)$ in the frontier of $V$   there is a continuous arc $\widehat \gamma: [0,1] \to D$ with $\widehat
\gamma(t) = \gamma(t)$ for $t \in [0,1).$ The point $\gamma(1)$ is
called an {\em accessible point} of the frontier of $V$ and $\widehat \gamma(1)$
is a prime end corresponding to it (there may be more than one prime
end corresponding to an accessible point).
\item[(iv)] If $\gamma_t$ is a continuous one parameter family of arcs in $S^2$ as in (iii) and if $\gamma(1)$ is independent of $t$ then $\widehat \gamma(1)$ is also independent of $t$.    
\end{itemize}

Properties (i)-(iii) go back to Caratheodory.  An excellent
modern exposition can be found in Mather's paper \cite{M}. In
particular see \S 17 of \cite{M} for a discussion of accessible points
and \S 18 for property (iv).

A properly embedded ray $R$ in $\ti S$ that is contained in $V_n$
determines an arc $\gamma$ as in (iii) with $\gamma(1) = \infty$.
Considering the rays $R$ and $G(R)$ and applying item (ii) of
Lemma~\ref{epsilon neighborhood is nice} we obtain $\gamma$ and
$\widehat G(\gamma)$ to which Property (iv) applies. 
This implies that $\widehat\gamma$ and
$\widehat G(\widehat\gamma)$ converge to the same prime end
$\widehat{\infty}$ which is evidently fixed by $\widehat G$.

Now choose a properly embedded ray $R_1$ in $H_{n+1}$ with initial
endpoint in $K_n$ and note that $R_1$ is disjoint from the frontier of
$H_n$.  Since $K_n$ is bounded, $R_1$ intersects the frontier of $K_n$
in some first point $z$ which is necessarily in the frontier of $V_n$.
The initial segment $\gamma_1$ of $R_1$ that terminates at $z$ satisfies
the hypotheses of (iii) with $\gamma(1) = z$.  Let $\hat \gamma$ be
the associated path in $D$ and let $\hat z \in \partial D$ be the
prime end $\hat \gamma(1).$
 
Our proof is by contradiction.  Let $m$ be the smallest natural number
with $G^m(K_n) = K_n$ and assume  $m \ne 1$ and $G(\bar K_n) \cap \bar K_n =
\emptyset$.  Then $\hat z, \widehat G(\hat z), \widehat G^m (\hat
z)$ and $\widehat G^{m+1} (\hat z)$ are distinct and in this order on
$\partial D \setminus \{\widehat{\infty}\}$ (oriented so that $\hat z 
< \widehat G(\hat z)$). Since
$G^m(K_n) = K_n$ the initial endpoints of the  $\widehat \gamma_1$ and $\widehat G^m(\widehat\gamma_1)$ can be joined in $K_n$ to form   an arc $\widehat \beta$ with interior in $K_n$ and endpoints $\hat z$ and $\widehat G^m(\hat z)$.      Therefore $\widehat G(\widehat \beta)$ has one endpoint $\widehat G(\hat z)$ and the other
$\widehat G^{m+1}(\hat z)$.  It follows that $\widehat G(\widehat \beta) 
\cap \widehat \beta \ne \emptyset$ and indeed the points of intersection
must lie in the interior of these arcs     in
contradiction to the assumption that $G(K_n) \cap K_n = \emptyset$.
We conclude that $m = 1$ and hence that $G(K_n) = K_n$.  \endproof

\end{document}